\documentclass[10pt]{amsart}
\usepackage{amssymb}
\usepackage{bm}
\usepackage{graphicx}
\usepackage[centertags]{amsmath}
\usepackage{amsfonts}
\usepackage{amsthm}
\usepackage{graphicx}

\newtheorem{thm}{Theorem}
\newtheorem{cor}[thm]{Corollary}
\newtheorem{lem}[thm]{Lemma}

\newtheorem{prop}[thm]{Proposition}

\newtheorem{defn}[thm]{Definition}
\theoremstyle{definition}
\newtheorem{rem}[thm]{Remark}

\newcommand{\rr}{\mathbb{R}}
\newcommand{\nn}{\mathbb{N}}
\newcommand{\ee}{\varepsilon}
\newcommand{\dd}{\delta}

\long\def\symbolfootnote[#1]#2{\begingroup%
\def\thefootnote{\fnsymbol{footnote}}\footnote[#1]{#2}\endgroup}

\begin{document}
\title[Operators on the Stopping Time space]{Operators on the Stopping Time space}

\author[D. Apatsidis]{Dimitris Apatsidis}
\address{National Technical University of Athens, Faculty of Applied Sciences,
Department of Mathematics, Zografou Campus, 157 80, Athens, Greece}
\email{dapatsidis@hotmail.com}

\begin{abstract}
Let $S^1$ be the stopping time space and $\mathcal{B}_1(S^1)$ be the Baire-1 elements of the second dual of $S^1$.
To each element $x^{**}$ in the space $\mathcal{B}_1(S^1)$ we
associate a positive Borel measure $\mu_{x^{**}}$ on the Cantor set.
We use the measures $\{\mu_{x^{**}}: x^{**}\in\mathcal{B}_1(S^1) \}$ to
characterize the operators $T:X\to S^1$, defined on a space $X$ with an unconditional basis, which preserve a copy of $S^1$.
In particular, we show that $T$ preserves a copy of $S^1$ if and only if the set $\{\mu_{x^{**}}:\;x^{**}\in\mathcal{B}_1(S^1)\}$
is non separable as a subset of $\mathcal{M}(2^\nn)$.
\end{abstract}

\maketitle

\symbolfootnote[0]{\textit{2010 Mathematics Subject
Classification:} Primary 46B03, Secondary 47B37, 46B09.}
\symbolfootnote[0]{\textit{Keywords and phrases:} operators, unconditional bases.}

\symbolfootnote[0]{This research was supported by program API$\Sigma$TEIA-1082.}

%-------------------------Introduction---------------------------%

\section{Introduction.}
The Stopping Time space $S^1$, was introduced by H. P. Rosenthal as the unconditional analogue of $L^1(2^\nn)$,
where by $2^\nn$ we denote the Cantor set and $L^1(2^\nn)$ is the Banach space of equivalence classes of
measurable functions on $2^\nn$ which are absolutely integrable on $2^\nn$, with respect to the Haar
measure. The space $S^1$ belongs to the
wider class of the spaces $S^p$, $1\leq p<\infty$ which we are
about to define. We denote by $2^{<\nn}$ the dyadic tree and by
$c_{00}(2^{<\nn})$ the vector space of all real valued functions
defined on $2^{<\nn}$ with finite support. For $1\leq p<\infty$ we
define the $\|\cdot\|_{S^p}$ norm on $c_{00}(2^{<\nn})$ as follows. For
$x\in c_{00}(2^{<\nn})$, we set
\[\|x\|_{S^p}=\sup\bigg(\sum_{s\in A}|x(s)|^p\bigg)^{1/p}\] where
the supremum is taken over all antichains $A$ of $2^{<\nn}$. The
space $S^p$ is the completion of $(c_{00}(2^{<\nn}),
\|\cdot\|_{S^p})$. The space $S^1$  has a 1-unconditional basis and G. Schechtman, in
an unpublished work, showed that it contains almost all $\ell_p$  isometrically, $1\leq
p<\infty$. This result was extended in \cite{B} to all $S^p$
spaces where it was shown that for every $p\leq q$, $\ell_q$ is embedded
into $S^p$. An excellent and detailed study of the stopping time
space $S^1$, in fact in a more general setting, is included in N.
Dew's Ph.D. Thesis \cite{D}. The interested reader will also find
therein,  among other things, a proof of Schechtman's unpublished
result. Also, in \cite{BO} H. Bang and E. Odell  showed that $S^1$
has the Weak Banach-Saks and the Dunford-Pettis properties, as the space $L^1(2^\nn)$.

In the sequel, by an operator between Banach spaces, we shall always mean a bounded linear operator.  Let $X$, $Y$, $Z$ be Banach spaces and $T:X\to Y$ be an operator.
We will say that the operator $T$, preserves a copy of $Z$, if there is a subspace $W$
of $X$ which is isomorphic to $Z$, such that the  restriction of $T$ on $W$  is an isomorphism.
The study of preserving properties of operators on $S^1$, is important to
the isomorphic classification of the complemented subspaces of $S^1$. The main results of the present paper are in this direction.

In \cite{ES} P. Enflo and T. W. Starbird showed that every
subspace $Y$ of $L^1$, isomorphic to $L^1$, contains a subspace which is isomorphic to $L^1$ and
complemented in $L^1$. From this it follows that if a complemented subspace $X$ of $L^1$ contains
a subspace isomorphic to $L^1$, then $X$ is isomorphic to $L^1$. Also, if $L^1$ is isomorphic to
a $\ell_1$ sum of a sequence of Banach spaces, then one of the spaces is isomorphic to $L^1$.
We will prove the same results in the case of $S^1$.

The work included in the present paper is motivated by the following problems.

\noindent\textit{Problem 1.}
Let $T:S^1\to S^1$ be an operator and $X$ a reflexive subspace
of $S^1$ such that the  restriction of $T$ on $X$  is an isomorphism.
Does $T$ preserve a copy of $S^1$?

\noindent\textit{Problem 2.}
Let $X$ be a complemented subspace of $S^1$ such that $\ell_1$ does not embed in $X$.
Is $X$ $c_0$-saturated?

It is well known that a subspace of a space with an unconditional basis, is reflexive if and only if it
does not contain  $c_0$ and $\ell_1$ isomorphically. Therefore, an affirmative answer to the first
problem yields an affirmative answer to the second one.

To state our main results we need some notation and to introduce some
terminology. If $X$ is a Banach space, by $\mathcal{B}_1(X)$ we shall denote
the space of all Baire-1 elements of $X^{**}$. That is, the space of all $x^{**}\in X^{**}$
such that there is a sequence $(x_n)_{n\in\nn}$ of $X$ with $x^{**}=\lim_nx_n$
in the weak star topology, which from \cite{OR} is equal to the space of all $x^{**}\in X^{**}$, so that $x^{**}|_{(B_{X^*},w^*)}$ is a Baire-1 function.
By $\mathcal{M}(2^\nn)$ we denote the space of all Borel measures on $2^\nn$
endowed by the norm $\|\mu\|=\sup\{|\mu(B)|: B\,\text{is a Borel subset of}\, 2^\nn\}$.
To each $x^{**}\in\mathcal{B}_1(S^1)$ we associate a positive
Borel measure on $2^\nn$ which will denote by $\mu_{x^{**}}$.
The definition of $\mu_{x^{**}}$ is related to a corresponding concept defined and studied in \cite{AAK} where the space $V_2^0$ is studied.
It is worth mentioning that in the case of $V_2^0$ we have
that $\mathcal{B}_1(V_2^0) = (V_2^0)^{**}$, i.e. a measure can be assigned to every $x^{**}\in (V_2^0)^{**}$.
This is not the case in the space $S^1$.

We are ready to state the first main result of the paper.

\begin{thm}\label{t1}
Let $X$ be a Banach space with an unconditional basis,
$Y$ a closed subspace of $X$ and $T: X \to S^1$ an
operator. The following assertions are equivalent.
\begin{itemize}

\item[(i)] The set $\{\mu_{T^{**}(y^{**})}: y^{**}\in\mathcal{B}_1(Y) \}$
is non-separable subset of $\mathcal{M}(2^\nn)$.

\item[(ii)] There exists a subspace $Z$ of
$Y$ isomorphic to $S^1$ such that the restriction of $T$ on $Z$  is an isomorphism.

\item[(iii)] There exists subspace $Z$ of
$Y$ isomorphic to $S^1$ such that the restriction of $T$ on $Z$  is an isomorphism and $T[Z]$ is complemented in $S^1$.

\end{itemize}
Moreover, if (iii) is satisfied then the space $Z$ is complemented in $X$.
\end{thm}

Clearly, since the basis of $S^1$ is $1$-unconditional, the above theorem also holds by replacing $X$ with the space $S^1$.

%\noindent{\bf{Theorem.}}
%Let $T: S^1 \to S^1$ be an
%operator and $Y$ a closed subspace of $S^1$.
% The following assertions are equivalent.
%\begin{itemize}

%\item[(i)] Then the set $\{\mu_{T^{**}(y^{**})}: y^{**}\in\mathcal{B}_1(Y) \}$
%is non-separable subset of $\mathcal{M}(2^\nn)$.

%\item[(ii)] There exists a subspace $Z$ of
%$Y$ isomorphic to $S^1$ such that the restriction of $T$ on $Z$  is an isomorphism.

%\item[(iii)] There exists subspace $W$ of
%$Y$ isomorphic to $S^1$ such that the restriction of $T$ on $W$  is an isomorphism and $T[W]$ is complemented in $S^1$.

%\end{itemize}
%Moreover, if (iii) is satisfied then the space $W$ is complemented in $S^1$.

We state some consequences of the above theorem.\\

\noindent{\bf{Corollary 1.}}
Let  $Y$ be a closed subspace of $S^1$. Then the set
$\mathcal{M}_{\mathcal{B}_1(Y)}=\{\mu_{y^{**}}: y^{**}\in\mathcal{B}_1(Y) \}$
is a non-separable subset of $\mathcal{M}(2^\nn)$, if and only if, there exists a subspace $Z$ of
$Y$ isomorphic to $S^1$ and complemented in $S^1$.\\

\noindent{\bf{Corollary 2.}}
Let $Y$ be a closed subspace of $S^1$ which is isomorphic to $S^1$.
Then $Y$ contains a subspace $Z$ isomorphic to $S^1$ which is complemented in $S^1$.\\

\noindent{\bf{Corollary 3.}}
Let $Y$ be a complemented subspace of $S^1$ such that the set
$\{\mu_{y^{**}}: y^{**}\in\mathcal{B}_1(Y) \}$ is a non-separable subset of $\mathcal{M}(2^\nn)$.
Then $Y$ is isomorphic to $S^1$.\\

In \cite{BR} it is shown that if $L^1$ $G_\delta$-embeds into $X$ where, $X$ is either isomorphic to
a dual space, or $X$ embeds into $L^1$, then $L^1$ embeds isomorphically into $X$.
We recall that a bounded and linear operator $T:X\to Y$ between Banach spaces, is called a
$G_\delta$-embedding, if it is injective and $T[K]$ is a $G_\delta$ subset of
$Y$, for all closed bounded $K$.\\

\noindent{\bf{Corollary 4.}}
Suppose $S^1$ is isomorphic to an $\ell_1$ sum of a sequence of Banach spaces $X_i$.
Then there is an $j$ such that $X_j$ is isomorphic to $S^1$.\\

Our second result is the following.\\

\noindent{\bf{Theorem 2.}}
Let $X$ be a Banach space. If $S^1$ $G_\delta$-embeds in $X$ and $X$ $G_\delta$-embeds in $S^1$,
then $S^1$ complementably embeds in $X$.\\

As a consequence of the above theorem we obtain the following\\

\noindent{\bf{Corollary 5.}}
Let $X$ be a closed subspace of $S^1$. If $S^1$ $G_\delta$-embeds in $X$,
then $S^1$ complementably embeds in $X$.\\

The paper is organized as follows. In section 2, we fix some notation that we shall use
and we remind the definition of $S^1$.
In section 3, we give the definition of the measure $\mu_{x^{**}}$,  while in section 4
we give the proofs of Theorems 1 and 2 and their corollaries.

\section{Preliminaries.}
\subsection{The dyadic tree.} For every $n\geq 0$,
we set $2^n=\{0,1\}^n$ (where $2^0=\{\varnothing\}$). Hence for
$n\geq 1$, every $s\in 2^n$ is of the form $s=(s(1),...,s(n))$.
For $0\leq m<n$ and $s\in 2^n$, we set $s|m=(s(1),...,s(m))$, where if
$m=0$, $s|0=\varnothing$. Also, $2^{\leqslant n}=\cup_{i=0}^ n 2^i$
and $2^{<\nn}=\cup_{n=0}^\infty 2^n$. The \textit{length} $|s|$ of
an $s\in 2^{<\nn}$, is the unique $n\geq 0$ such that $s\in 2^n$.
The initial segment partial ordering on $2^{<\nn}$ will be denoted
by $\sqsubseteq$ (i.e. $s\sqsubseteq t$ if $m=|s|\leq |t|$ and
$s=t|m$). For $s, t\in 2^{<\nn}$, $s\perp t$ means that $s,t$ are
$\sqsubseteq$-incomparable (that is neither $s\sqsubseteq t$ nor
$t\sqsubseteq s$). For an $s\in 2^{<\nn}$, $s^\smallfrown 0$ and
$s^\smallfrown 1$ denote the two immediate successors of $s$ which
end with $0$ and $1$ respectively.

An \textit{antichain} of $2^{<\nn}$, is a subset of $2^{<\nn}$
such that for every $s,t\in A$, $s\perp t$.
If $A,B$ are subsets of $2^{<\nn}$
then we write $A\perp B$ if for all $s\in A$ and for all
$t\in B$, $s\perp t$. By $\mathcal{A}$ we denote the set of all antichain of $2^{<\nn}$.
A \textit{branch} of
$2^{< \nn}$ is a maximal totally ordered subset of $2^{< \nn}$.

A subset $\mathcal{I}$ of $2^{<\nn}$ is called a
\textit{segment} if $(\mathcal{I},\sqsubseteq)$ is linearly
ordered by $\sqsubseteq$ and moreover for every
$s\sqsubset v\sqsubset t$, $v$ is contained in
$\mathcal{I}$  provided that   $s,t$ belong to
$\mathcal{I}$. For a finite segment $\mathcal{I}$ of $2^{<\nn}$, by
$\max\mathcal{I}$ we denote the $\sqsubseteq$-greatest element of $\mathcal{I}$.

A segment $\mathcal{I}$ is called \textit{initial} if the empty
sequence $\varnothing$ belongs to $\mathcal{I}$. For any
 $s\in 2^{<\nn}$, let $\mathcal{I}(s)=\{t\in2^{<\nn}:\; t\sqsubseteq s\}.$ Then clearly
 $\mathcal{I}(s)$ is an initial segment of $2^{<\nn}$. For $s,t\in 2^{<\nn}$,
 the \textit{infimum} of $\{s,t\}$ is  defined by $s\wedge t=\max(\mathcal{I}(s)\cap\mathcal{I}(t))$.

The \textit{lexicographical ordering} of $2^{<\nn}$, denoted by $\leq_{lex}$
is defined as follows. For every $s,t\in 2^{<\nn}$, $s\leq_{lex}t$ if either
$s\sqsubseteq t$ or $s\perp t$, $w^\frown 0\sqsubseteq s$ and $w^\frown 1\sqsubseteq t$ where
$w=s\wedge t$. Also we write $s<_{lex}t$ if $s\leq_{lex}t$ and $s\neq t$.
The lexicographical ordering is a total ordering of $2^{<\nn}$.
This ordering of $2^{<\nn}$ (which identifies $2^{<\nn}$
with $\nn$) will be called the \textit{natural ordering} of
$2^{<\nn}$. According to the above, we can write $2^{<\nn}$ in the form of a sequence $(s_n)_{n\in\nn}$, where $n<m$ if either
$|s_n|<|s_m|$ or $|s_n|=|s_m|$ and $s_n <_{lex} s_m$.

 A
\textit{dyadic subtree} is a subset $T$ of $2^{<\nn}$ such that
there is an order isomorphism $\phi:2^{<\nn}\to T$. In this case
$T$ is denoted by $T=(t_s)_{s\in 2^{<\nn}}$, where $t_s=\phi(s)$.
A dyadic subtree $(t_s)_{s\in 2^{<\nn}}$ is said to be \textit{regular}
if for every $n\in\nn$ there exists $m\in\nn$ such that $\{t_s: s\in 2^n\}\subseteq 2^m$.

Let $2^{\nn}$ be the Cantor space, i.e., the set of all infinite sequences $\sigma=(\sigma(n))_n$
consisting of elements of  $2=\{0,1\}$. If $\sigma\in 2^{\nn}$ and $n\in\nn$,
let $\sigma|n=(\sigma(1),\ldots,\sigma(n))\in 2^n$. We say that $s\in 2^n$
is an \textit{initial segment} of $\sigma\in 2^{\nn}$ if $s=\sigma|n$. We
write $s\sqsubseteq \sigma$ if $s$ is an initial segment of $\sigma$.

\subsection{The stopping time space $S^1$}The space $S^1$ is defined to be the completion of
$c_{00}(2^{<\nn})$ under the norm
\[\Big\|\sum_{s}\lambda_s e_s\Big\|_{S^1}=sup\Big( \sum_{s\in
A}|\lambda_s|\Big)\] where the supremum is taken over all
antichains $A$ of $2^{<\nn}$. For every $s\in 2^{<\nn}$, by $e_s$ we denote the
characteristic function of $\{s\}$. By the definition of the $S^1$-norm, we observe that
the family $\{e_s\}_{s\in 2^{<\nn}}$
is a 1-unconditional Schauder basis of $S^1$.
Also, for every infinite chain $C$ of $2^{<\nn}$ the subspace of $S^1$
generated by $\{e_s:s\in C\}$ is isomorphic to $c_0$, while for
every infinite antichain $A$ the corresponding one is isomorphic
to $\ell_1$.

\section{measures associated to elements of Baire-1 space.}
Our general notation and terminology is standard and it can be found, for instance in \cite{LT}.
Let $X$ be a Banach space. As we have mentioned in the introduction,
by $\mathcal{B}_1(X)$ we denote the space of all $x^{**}\in X^{**}$
such that there is a sequence $(x_n)_{n\in\nn}$ of $X$ with $x^{**}=\lim_nx_n$
in the weak star topology.
The aim of  this section is to  introduce and study the
fundamental properties of the measure $\mu_{x^{**}}$ corresponding to an element
$x^{**}\in \mathcal{B}_1(S^1)$.
We start with some general results about a Banach space with an unconditional basis.

\begin{lem}\label{leq}
Let $X$ be a Banach space with an unconditional normalized basis $(e_i)_{i\in\nn}$
and $x^{**}\in \mathcal{B}_1(X)$ such that $x^{**}(e_i^*)=0$,
for every $i\in\nn$. Then $x^{**}=0$.
\end{lem}

\begin{proof}
 Since $x^{**}\in \mathcal{B}_1(X)$
there exists a sequence $(x_n)_n$ of elements of $X$ such that
$x^{**}=\lim_n x_n$ in the weak star topology and $\lim_n e_i^*(x_n)=0$,
for every $i\in\nn$. By a classical result of Bessaga and Pelczynski , we may assume that  $(x_n)_n$
is equivalent to a block basic sequence taken with respect to $(e_i)_i$.
Since $(e_i)_i$ is an unconditional basis we have that every block sequence of $(e_i)_i$
is unconditional. But then, either $(x_n)_n$ has a subsequence which is equivalent
to the usual basis of $\ell_1$ or is weakly null. Since the basis of $\ell_1$ is not
weakly Cauchy, we have that $x^{**}=0$.
\end{proof}

\begin{prop}\label{pb1}
Let $X$ be a Banach space with a $1$-unconditional normalized basis $(e_i)_{i=1}^\infty$.
Then the space $\mathcal{B}_1(X)$ can be identified with the space of all
sequence of scalars $(a_i)_i$ such that
$\sup_n\|\sum_{i=1}^n a_i e_i\|<\infty$.
This correspondence is given by
$\mathcal{B}_1(X) \ni x^{**}\leftrightarrow (x^{**}(e_i^*))_i$.
Moreover, for every $x^{**}\in\mathcal{B}_1(X)$ we have that
$x^{**}=\lim_n\sum_{i=1}^nx^{**}(e_i^*) e_i$
in the weak star topology and
\begin{equation}\label{pb11}
\|x^{**}\|=\sup_n\bigg\|\sum_{i=1}^n x^{**}(e_i^*)e_i\bigg\|.
\end{equation}
\end{prop}

\begin{proof}
Let $x^{**}\in \mathcal{B}_1(X)$ and $x^*\in X^*$. For $n\in\nn$,
we choose a sequence $(\ee_i)_{i=1}^n$ of signs so that
$\sum_{i=1}^n |x^{**}(e_i^*)x^*(e_i)|=\sum_{i=1}^n \ee_ix^{**}(e_i^*)x^*(e_i)$.
Since $(e_i)_i$ is a $1$-unconditional normalized basis of $X$,
we have that $\|\sum_{i=1}^n \ee_ix^*(e_i)e_i\|\leq \|x^*\|$ and
\[\sum_{i=1}^n |x^{**}(e_i^*)x^*(e_i)|=\sum_{i=1}^n \ee_ix^{**}(e_i^*)x^*(e_i)=x^{**}
\bigg(\sum_{i=1}^n \ee_ix^*(e_i)e_i\bigg)\leq\|x^{**}\|\|x^*\|.\]
Hence, we obtain that the sequence $(\sum_{i=1}^nx^{**}(e_i^*) e_i)_n$ is weakly Cauchy and
\[\bigg\|\sum_{i=1}^n x^{**}(e_i^*)e_i\bigg\|\leq\|x^{**}\|,\,\, \text{for all}\,\, n\in\nn.\]
Let $y^{**}=w^*-\lim_n\sum_{i=1}^nx^{**}(e_i^*) e_i$.
Then we observe that  $(y^{**}-x^{**})(e_i^*)=0$, for every $i\in\nn$
and $y^{**}-x^{**}\in \mathcal{B}_1(X)$.
By Lemma \ref{leq}, we get that $y^{**}=x^{**}$. Therefore,
$x^{**}=w^*-\lim_n\sum_{i=1}^nx^{**}(e_i^*) e_i$.
By the above inequality and since $(e_i)_i$ is a monotone basis of $X$,
from the weak star lower semicontinuity of the second dual norm  we have that
\[\|x^{**}\|=\lim_n\Big\|\sum_{i=1}^n x^{**}(e_i^*)e_i\Big\|=
\sup_n\Big\|\sum_{i=1}^n x^{**}(e_i^*)e_i\Big\|.\]

Conversely, let $(a_i)_i$ be such that
$C=\sup_n\|\sum_{i=1}^n a_i e_i\|<\infty$. For every $x^*\in X^*$ and $n\in\nn$,
we choose a sequence $(\ee_i)_{i=1}^n$ of signs such that
$\sum_{i=1}^n |a_i x^*(e_i)|=x^*(\sum_{i=1}^n \ee_i a_i e_i)$.
As $(e_i)_i$ is $1$-unconditional normalized basis of $X$, we get that
$\sum_{i=1}^n |a_i x^*(e_i)|\leq C\|x^*\|$. Therefore, the sequence $(\sum_{i=1}^n a_i x^*(e_i))_n$
is weakly Cauchy. If $x^{**}=w^*-\lim_n \sum_{i=1}^n a_i e_i$,
by Lemma \ref{leq}, we get that $x^{**}$ is the unique $x^{**}\in \mathcal{B}_1(X)$ such that
and $x^{**}(e_i^*)=a_i$, for every $i\in\nn$.
\end{proof}

If $X$ is a Banach space and $M$ a subset of $X$, then the closed linear span of $M$,
denoted by $[M]$, is the smallest closed subspace of $X$ containing the set $M$.

\begin{prop}\label{dual}
Let $X$ be a Banach space with a $1$-unconditional normalized basis $(e_i)_{i=1}^\infty$.
Then the operator $J: \mathcal{B}_1(X)\to [(e_i^*)_i]^*$ with $J(x^{**})= x^{**}|_{[(e_i^*)_i]}$
is an isometry and onto.
\end{prop}

\begin{proof}
The fact that the operator $J$ is an isometry, is an immediate consequence of (\ref{pb11}).
We proceed to show that $J$ is onto. Let $f\in [(e_i^*)_i]^*$. Then we easily observe that
$\sup_n\|\sum_{i=1}^n f(e_i^*)e_i\|\leq\|f\|$. Hence by Proposition \ref{pb1}, there is $x^{**}\in\mathcal{B}_1(X)$
such that $x^{**}=w^*-\lim_n\sum_{i=1}^n f(e_i^*)e_i$. Therefore $x^{**}(e_i^*)=f(e_i^*)$, for all $i$ and so,
$J(x^{**})=x^{**}|_{[(e_i^*)_i]}=f$.
\end{proof}

\begin{rem}\label{r1}
By Proposition \ref{pb1}, we observe that for every $x^{**}\in\mathcal{B}_1(S^1)$,
\[\|x^{**}\|=\sup_{n\geq 0}\bigg\|\sum_{|s|\leq n} x^{**}(e_s^*)e_s\bigg\|=\sup_{A\in\mathcal{A}}\sum_{s\in A}|x^{**}(e_s^*)|.\]
\end{rem}

We are now ready to give the definition of $\mu_{x^{**}}$, the measure associated with an element $x^{**}$ in $\mathcal{B}_1(S^1)$. We first introduce some simple notation. For every  $x^{**}\in \mathcal{B}_1(S^1)$ and for every $D\subseteq 2^{<\nn}$ we set
\begin{equation*}
x^{**}|D(e_s^*)=
\begin{cases}
x^{**}(e^*_s),   & \text{if}\,\,\,s\in D\\
0, &  \text{if}\,\,\, s\in 2^{<\nn}\setminus D.
\end{cases}
\end{equation*}
By Remark \ref{r1}, we easily observe that $x^{**}|D\in\mathcal{B}_1(S^1)$ and for every
$D_1, D_2\subseteq 2^{<\nn}$ with $D_1\subseteq D_2$,
\begin{equation}\label{equat1}
\|x^{**}|D_1\|\leq \|x^{**}|D_2\|\leq\|x^{**}\|,
\end{equation}
and if $D_1\perp D_2$, then
\begin{equation}\label{equat2}
\|x^{**}|D_1\cup D_2\|=\|x^{**}|D_1\|+\|x^{**}|D_2\|.
\end{equation}

For every $t\in 2^{<\nn}$ we set $V_t=\{\sigma\in 2^{\nn}: t\sqsubseteq\sigma\}$. Recall that $\{V_t:\;t\in 2^{<\nn}\}$ is the usual basis of $2^{\nn}$, consisting of clopen sets. Also, for $t\in 2^{<\nn}$ and $m\geq 0$ we set
\[T_t^m=\{s\in 2^{<\nn}: t\sqsubseteq s,\, |s|\geq m\}.\]
By (\ref{equat2}), it is easy to see that for every $t\in 2^{<\nn}$,
$$\inf_m\|x^{**}|T_t^m\| = \inf_m\|x^{**}|T_{t^\smallfrown 0}^m\| + \inf_m\|x^{**}|T_{t^\smallfrown 1}^m\|. $$
%\[\widetilde{\mu}_{(x_n^{**})}(V_t)=
%\widetilde{\mu}_{(x_n^{**})}(V_{t^\smallfrown 0})+\widetilde{\mu}_{(x_n^{**})}(V_{t^\smallfrown 1}).\]
Therefore, by a classical result of Caratheodory, there exists a unique finite positive Borel measure $\mu_{x^{**}}$ on $2^{\nn}$ with
$\mu_{x^{**}}(V_t) = \inf_m\|x^{**}|T_t^m\|$, for every $t\in 2^{<\nn}$.

\begin{defn}
 Let $x^{**}\in\mathcal{B}_1(S^1)$. We define ${\mu}_{x^{**}}$ to be the unique finite positive Borel measure on $2^{\nn}$ so that for all $t\in 2^{\nn}$,
% \[\widetilde{\mu}_{x^{**}}(V_t)=\inf_m\|x^{**}|T_t^m\|.\]
\[{\mu}_{x^{**}}(V_t)=\inf_m\|x^{**}|T_t^m\|.\]
\end{defn}

\begin{rem}
The technique used, given an $x^{**}\in\mathcal{B}_1(S^1)$, to define the measure $\mu_{x^{**}}$, is among the same lines as the one used in \cite{AAK}.
\end{rem}

\begin{rem}\label{r2}
Let $x^{**}\in\mathcal{B}_1(S^1)$. By (\ref{equat2}) we observe that for every finite antichain $A$ of $2^{<\nn}$
and $m\geq 0$,
$\|x^{**}|\cup_{t\in A}T_t^m\|=\sum_{t\in A}\|x^{**}|T_t^m\|$.
Therefore, $\mu_{x^{**}}(\cup_{t\in A}V_t)=\inf_m\|x^{**}|\cup_{t\in A}T_t^m\|$,
for every finite antichain $A$ of $2^{<\nn}$.
\end{rem}

The following proposition is easily established.

\begin{prop}\label{basprop}
The following hold.
\begin{enumerate}
\item[(i)] For every
$x^{**}\in\mathcal{B}_1(S^1)$ and $\lambda\in\rr$, $\mu_{\lambda
x^{**}}=|\lambda|\mu_{x^{**}}$.
\item[(ii)]  For every $x^{**}, y^{**}\in\mathcal{B}_1(S^1)$, $\mu_{x^{**}+y^{**}} \leq \mu_{x^{**}} + \mu_{y^{**}}$.
\item [(iii)] For every $x^{**}\in\mathcal{B}_1(S^1)$,  $\mu_{x^{**}}(2^\nn)=d(x^{**}, S^1)=\inf\{\|x^{**}-x\|: x\in S^1\}$.
\item [(iv)] For every $x^{**}, y^{**}\in\mathcal{B}_1(S^1)$,
$\|\mu_{x^{**}}-\mu_{y^{**}}\|\leq \|x^{**}-y^{**}\|$.

\end{enumerate}
\end{prop}

\begin{prop}\label{property}
Let $x^{**}\in\mathcal{B}_1(S^1)$. Then for every antichain $A$ of $2^{<\nn}$,
\[\mu_{x^{**}}(\cup_{t\in A}V_t)=\inf_m\|x^{**}|\cup_{t\in A}T_t^m\|.\]
\end{prop}

\begin{proof}
Let $A$ be an antichain of $2^{<\nn}$. If $A$ is finite, then the conclusion follows
by Remark \ref{r2}. We assume that $A$ is infinite and let $A=\{t_i\}_{i=1}^\infty$ be an enumeration of $A$.
Since  $\lim_j\mu_{x^{**}}(\cup_{i=1}^jV_{t_i})=\mu_{x^{**}}(\cup_{t\in A}V_t)$,
we have that
\begin{equation}\label{property1}
\mu_{x^{**}}(\cup_{t\in A}V_t)=\lim_j\inf_m\|x^{**}|\cup_{i=1}^jT_{t_i}^m\|
\leq\inf_m\|x^{**}|\cup_{t\in A}T_t^m\|.
\end{equation}
We claim that $\lim_j\|x^{**}|\cup_{i=j}^\infty T_{t_i}^0\|=0$. Indeed,
if the claim is not true, by Remark \ref{r1} it is easy to see that there exist
$\ee>0$, a strictly increasing sequence $1\leq j_1<j_2<...$ in $\nn$ and a sequence
$(A_k)_{k=1}^\infty$ of finite antichains of $2^{<\nn}$, such that the following hold.
\begin{enumerate}
\item[(a)] For every $k\geq 1$, $A_k\subseteq\cup_{i=j_k+1}^{j_k}T_{t_i}^0$.
\item[(b)] For every $k\geq 1$, $\sum_{s\in A_k}|x^{**}(e^*_s)|>\ee$.
\end{enumerate}
Clearly the set $B=\cup_{k=1}^\infty A_k$ is an antichain. Hence by (b), we get that
\[\|x^{**}\|\geq\sum_{s\in B}|x^{**}(e^*_s)|=\sum_{k=1}^\infty\sum_{s\in A_k}|x^{**}(e^*_s)|=\infty,\]
a contradiction.

Let $\ee>0$. By the claim, we choose $j_0\in\nn$ with $\|x^{**}|\cup_{i=j_0+1}^\infty T_{t_i}^0\|<\ee$.
Then we obtain that
\[\begin{split}
\inf_m\|x^{**}|\cup_{i=1}^\infty T_{t_i}^m\|&=
\inf_m\|x^{**}|\cup_{i=1}^{j_0} T_{t_i}^m\|+\inf_m\|x^{**}|\cup_{i=j_0+1}^\infty T_{t_i}^m\|\\&\leq
\mu_{x^{**}}(\cup_{i=1}^{j_0}V_{t_i})+\|x^{**}|\cup_{i=j_0+1}^\infty T_{t_i}^0\|\\&<
\mu_{x^{**}}(\cup_{i=1}^{j_0}V_{t_i})+\ee \leq\mu_{x^{**}}(\cup_{i=1}^\infty V_{t_i})+\ee.
\end{split}\]
Therefore,
\begin{equation}\label{property2}
\inf_m\|x^{**}|\cup_{t\in A} T_t^m\|\leq \mu_{x^{**}}(\cup_{t\in A}V_{t_i}).
\end{equation}
By (\ref{property1}) and (\ref{property2}), the conclusion of proposition follows.
\end{proof}

\section{operators preserving a copy of $S^1$.}

This section is devoted to the proof of Theorems 1 and 2 and their corollaries. The next lemma can be found in \cite[Lemma 41, page 4252]{AAK}.
\begin{lem}\label{ineq}
Let $(\alpha_s)_{s\in 2^{<\nn}}$ , $(\lambda_s)_{s\in 2^{<\nn}}$
be families of non negative real numbers and let $n\geq 0$. Then
there exists a maximal  antichain $A$ of $ 2^{\leqslant n}$ and a
family of branches $(b_t)_{t\in A}$ of $2^{\leqslant n}$ such that
 $\sum_{|s|\leq n}\lambda_s\alpha_s\leq \sum_{t\in
A}(\sum_{s\in b_t}\alpha_s)\lambda_t$ and $t\in b_t$, for all
$t\in A$. Therefore  if $\sum_{n=1}^{\infty}\alpha_{\sigma|n}\leq
C$,  for all $\sigma\in 2^{\nn}$, then for each  $n\geq 0$ there
 is an antichain $A$ of $2^{\leqslant n}$ such that
$\sum_{|s|\leq n}\lambda_s\alpha_s\leq C\sum_{s\in A}\lambda_s$.
\end{lem}

\begin{prop}\label{eq}
Let $X$ be a Banach space and $(x_s)_{s\in 2^{<\nn}}$ be a family of
elements of $X$ with the following properties.
 \begin{enumerate}
\item[(i)] $(x_s)_{s\in 2^{<\nn}}$ is a $K$-unconditional basic family.
\item[(ii)] There is a constant $c>0$ such that for each finite antichain $A$ of $2^{<\nn}$,
and each family $(\lambda_s)_{s\in A}$ of scalars, we have
$\|\sum_{s\in A}\lambda_sx_s\|\geq c\sum_{s\in A}|\lambda_s|$.
\item[(iii)] There is a constant $C>0$ such
that for every  $\sigma\in 2^\nn$, $n\geq 0$ and every sequence of scalars $(a_k)_{k=0}^n$,
$\|\sum_{k=0}^n a_k x_{\sigma|k}\|\leq C\max_{0\leq k\leq n}|a_k|$.
\end{enumerate}
Then the family $(x_s)_{s\in 2^{<\nn}}$ is
equivalent to the basis $(e_s)_{s\in 2^{<\nn}}$ of $S^1$. In particular for every $n\geq 0$
and every family of scalars $(\lambda_s)_{|s|\leq n}$,
\[\frac{c}{K}\bigg\|\sum_{|s|\leq n}\lambda_se_s\bigg\|\leq \bigg\|\sum_{|s|\leq n}\lambda_sx_s\bigg\|
\leq C\bigg\|\sum_{|s|\leq n}\lambda_se_s\bigg\|.\]
\end{prop}

\begin{proof}
First we show the upper $S^1$-estimate. Fix a family of scalars $(\lambda_s)_{|s|\leq n}$.
Let $x^*\in (S^1)^*$, $\|x^*\|=1$.
By Lemma \ref{ineq}, (with $|x^*(x_s)|$ in place of $\alpha_s$),
there is an antichain $A\subseteq 2^{\leqslant n}$ and a
family of branches $(b_t)_{t\in A}$ of $2^{\leqslant n}$ such that
\begin{equation}\label{eq1}
\bigg|\sum_{|s|\leq n}\lambda_s x^*(x_s)\bigg|\leq\sum_{|s|\leq n}|\lambda_s| |x^*(x_s)|\leq
\sum_{t\in A}\bigg(\sum_{s\in b_t} |x^*(x_s)|\bigg)|\lambda_t|.
\end{equation}
By (iii), we have that for every $t\in A$, $\sum_{s\in b_t} |x^*(x_s)|\leq C$.
Hence by (\ref{eq1}), we obtain that
\[\bigg|\sum_{|s|\leq n}\lambda_s x^*(x_s)\bigg|\leq C\sum_{t\in A}|\lambda_t|.\]
This yields that $\|\sum_{|s|\leq n}\lambda_sx_s\|\leq C\|\sum_{|s|\leq n}\lambda_se_s\|$.

We now proceed to show the lower $S^1$-estimate. Let $A$ be an antichain of $2^{\leqslant n}$ such that
$\|\sum_{|s|\leq n}\lambda_se_s\|=\sum_{s\in A}|\lambda_s|$. Then by (i) and (ii), we obtain that
\[\frac{c}{K}\bigg\|\sum_{|s|\leq n}\lambda_se_s\bigg\|=
\frac{c}{K}\sum_{s\in A}|\lambda_s|\leq\frac{1}{K}\bigg\|\sum_{s\in A}\lambda_sx_s\bigg\|
\leq\bigg\|\sum_{|s|\leq n}\lambda_sx_s\bigg\|.\]
\end{proof}

\begin{rem}\label{r3}
By Proposition \ref{r3}, it is easy to see that if $(x_s)_s$ is a family in $X$
equivalent to the $S^1$-basis and $(t_s)_s$ is a dyadic subtree of $2^{<\nn}$,
then the family $(x_{t_s})_s$ is equivalent to the $S^1$-basis as well.
\end{rem}

\begin{lem}\label{lmes}
Let $x^{**}\in\mathcal{B}_1(S^1)$ and $(x_n)_n$ be a sequence in $S^1$ which is ${weak}^*$ convergent to $x^{**}$.
If there exists an antichain $B$ of $2^{<\nn}$ such that $\mu_{x^{**}}(\cup_{s\in B} V_s)>\rho>0$, then
for every $k\in\nn$, there exist a finite antichain $A$ of $2^{<\nn}$ with $A\subseteq\cup_{s\in B}T_s$
and $l>k$ such that
$\sum_{s\in A}|e_s^*(x_l-x_k)|>\rho$.
\end{lem}

\begin{proof}
Let $k\in\nn$ and $0<3\ee<\mu_{x^{**}}(\cup_{s\in B}V_s)-\rho$.
Since $x_k\in S^1$, there is some $m\in\nn$ such that
\begin{equation}\label{lmes1}
\|x_k|T_\varnothing^m\|<\ee.
\end{equation}
Choose a finite antichain $A$ of $2^{<\nn}$ with $A\subseteq\cup_{s\in B}T_s^m$
such that
\begin{equation}\label{lmes2}
\mu_{x^{**}}(\cup_{s\in B} V_s)-\ee<\sum_{s\in A}|x^{**}(e_s^*)|
\end{equation}
Moreover since
$(x_n)_n$ converges ${weak}^*$ to $x^{**}$, there is $l>k$ such that
\begin{equation}\label{lmes3}
\bigg|\sum_{s\in A}|x^{**}(e_s^*)|-\sum_{s\in A}|e_s^*(x_n)|\bigg|<\ee.
\end{equation}
Then by (\ref{lmes1}), (\ref{lmes2}) and (\ref{lmes3}) we have that
\[\begin{split}
\sum_{s\in A}|e_s^*(x_l-x_k)|&\geq
\sum_{s\in A}|e_s^*(x_l)|-\sum_{s\in A}|e_s^*(x_k)|>
\sum_{s\in A}|x^{**}(e_s^*)|-2\ee\\&>
\mu_{x^{**}}(\cup_{s\in B} V_s)-3\ee>\rho.
\end{split}\]
\end{proof}

In the sequel, if $Y$ is a subspace of the Banach space $X$,
then we identify $Y^{**}$ with the subspace $Y^{\perp\perp}$ in $X^{**}$.

\begin{lem}\label{fact}
Let $X$ be a Banach space with a $1$-unconditional normalized basis $(e_i)_i$,
$Y$ a closed subspace of $X$ and $T: X \to S^1$ be a
bounded  linear operator. We suppose that there exist $y^{**}\in \mathcal{B}_1(Y)$ and
an antichain $B$ of $2^{<\nn}$ such that \[\mu_{T^{**}(y^{**})}(\cup_{s\in B} V_s)>\rho>0.\]
Then for every $m\in\nn$ and $\ee>0$, there exist $y\in Y$, a finite subset $F$ of  $\nn$, scalars $\{d_i\}_{i\in F}$,
with $m<\min F$, $0\leq d_i\leq 1$ for every $i\in F$ and a finite antichain $A$ of $2^{<\nn}$ with $A\subseteq\bigcup_{s\in B}T_s$
such that \[\bigg\|y-\sum_{i\in F}d_iy^{**}(e_i^*)e_i\bigg\|<\ee\,\, \text{and}\,\, \sum_{s\in A}|e_s^*(T(y))|>\rho.\]
\end{lem}

\begin{proof}
We fix $m\in\nn$ and $\ee>0$. Let $(y_n)_n$ be a sequence in $Y$ which is ${weak}^*$ convergent to $y^{**}$.
By Proposition \ref{pb1}, the sequence $(\sum_{i=1}^jy^{**}(e_i^*)e_i)_j$ ${weak}^*$ convergent to $y^{**}$.
Hence the sequence  $(y_j-\sum_{i=1}^jy^{**}(e_i^*)e_i)_j$ is weakly null. By Mazur's theorem there is
a convex block sequence $(\sum_{j\in F_n}r_j(y_j-\sum_{i=1}^jy^{**}(e_i^*)e_i))_n$, where $(F_n)_n$ is a
sequence of successive finite subsets of $\nn$ (i.e., $\max F_n<\min F_{n+1}$ for all $n\in\nn$),
$\sum_{j\in F_n}r_j=1$ and $r_j\geq 0$, for all $j\in F_n$, such that
\[\lim_n\bigg\|\sum_{j\in F_n}r_j(y_j-\sum_{i=1}^jy^{**}(e_i^*)e_i)\bigg\|=0.\]
We choose $n_1\in\nn$, with $m<\min F_{n_1}$ and
\begin{equation}\label{fact1a}
\bigg\|\sum_{j\in F_n}r_jy_j-\sum_{j\in F_n}r_j\sum_{i=1}^jy^{**}(e_i^*)e_i\bigg\|<\ee/2,\,\, \text{for all}\,\, n\geq n_1.
\end{equation}
Since the sequence $(\sum_{j\in F_n}r_jy_j)_n$ is weak star convergent to $y^{**}$, $T^{**}$ is ${weak}^*$-${weak}^*$ continuous
and $T^{**}|X=T$, we have that the sequence $(T(\sum_{j\in F_n}r_jy_j))_n$ weak star convergent to $T^{**}(y^{**})$.
Applying  Lemma \ref{lmes} we obtain that there exist a finite antichain $A$ of $2^{<\nn}$ with $A\subseteq\bigcup_{s\in B}T^0_s$
and $n_2>n_1$, such that
\begin{equation}\label{fact2a}
\sum_{s\in A}\bigg|e_s^*\bigg(T\bigg(\sum_{j\in F_{n_2}}r_jy_j-\sum_{j\in F_{n_1}}r_jy_j\bigg)\bigg)\bigg|>\rho.
\end{equation}
We set $y_1=\sum_{j\in F_{n_1}}r_jy_j$, $y_2=\sum_{j\in F_{n_2}}r_jy_j$ and $y=y_2-y_1$. Note that
\[\sum_{j\in F_{n_2}}r_j\sum_{i=1}^jy^{**}(e_i^*)e_i-\sum_{j\in F_{n_1}}r_j
\sum_{i=1}^jy^{**}(e_i^*)e_i=\sum_{i\in F}d_iy^{**}(e_i^*)e_i,\]
where $F$ is a finite subset of $\nn$, $m<\min F$ and $0\leq d_i\leq 1$, for every $i\in F$. Therefore by
\eqref{fact1a} and \eqref{fact2a} the conclusion of the lemma follows.
\end{proof}

We remind that $\mathcal{M}^+(2^\nn)$ denotes the positive cone of $\mathcal{M}(2^\nn)$.

\begin{lem} \label{fact1}Let $\{\mu_\xi\}_{\xi<\omega_1}$ be a non-separable  subset of $\mathcal{M}^+(2^\nn)$.
Then there is an uncountable subset $\Gamma$  of $\omega_1$ such
that for every $\xi\in \Gamma$,
 $\mu_\xi=\lambda_\xi+\tau_\xi$  where $\lambda_\xi,\tau_\xi$ are positive Borel measures on $2^\nn$
  satisfying the following properties.
\begin{enumerate}
\item For all $\xi\in \Gamma$,  $\lambda_\xi\perp\tau_\xi$ and
$\|\tau_\xi\|>0$. \item For all $\zeta<\xi$ in $\Gamma$,
$\mu_\zeta\perp \tau_\xi$.
\end{enumerate}
In particular, for all $\zeta<\xi$ in $\Gamma$,
$\tau_\zeta\perp \tau_\xi$.
\end{lem}
\begin{proof} We may suppose that for some $\delta>0$,
 $\|\mu_\xi-\mu_\zeta\|>\delta$, for all $0\leq \zeta<\xi<\omega_1$. By transfinite induction we
 construct a strictly increasing sequence $(\xi_\alpha)_{\alpha<\omega_1}$ in $\omega_1$ such that
 for each $\alpha<\omega_1$, $\mu_{\xi_\alpha}=\lambda_{\xi_\alpha}+\tau_{\xi_\alpha}$, with
  $\lambda_{\xi_\alpha}\perp \tau_{\xi_\alpha}$, $\|\tau_{\xi_\alpha}\|>0$ and  $\tau_{\xi_\alpha}\perp \mu_{\xi_\beta}$ for all $\beta<\alpha$.
  The general inductive step of the construction is as follows. Suppose that for some $\alpha<\omega_1$,
$(\xi_\beta)_{\beta<\alpha}$ has been defined. Let $(\beta_n)_n$ be an enumeration of $\alpha$ and set
 \[\zeta_\alpha=\sup_n\xi_{\beta_n},\;\;
 \nu_\alpha=\sum_n\mu_{\xi_{\beta_n}}/2^n\;\text{and}\;N_\alpha=\{\xi<\omega_1:\zeta_\alpha<\xi\;\text{and}\;\mu_\xi<<\nu_\alpha\}\]
By the Radon - Nikodym theorem, $\{\mu_\xi\}_{\xi\in N_\alpha}$ is
isometrically contained in $L_1(2^\nn,\nu_\alpha)$ and therefore
it is norm separable. Since we have assumed that
$\|\mu_\xi-\mu_\zeta\|>\delta$, for all $0\leq
\zeta<\xi<\omega_1$, we get that $N_\alpha$ is countable. Hence we
can choose $\xi_\alpha>\sup N_\alpha$. Let
$\mu_{\xi_\alpha}=\lambda_{\xi_\alpha}+\tau_{\xi_\alpha}$ be the
Lebesgue analysis of $\mu_{\xi_\alpha}$ where
$\lambda_{\xi_\alpha}<<\nu_{\alpha} $ and $\tau_{\xi_\alpha}\perp
\nu_{\alpha}$. By the definition of $\nu_{\alpha}$
 and $\xi_\alpha$, we have that $\|\tau_{\xi_\alpha}\|>0$, $\tau_{\xi_\alpha}\perp\mu_{\xi_\beta}$,
 for all $\beta<\alpha$ and the inductive step
 of the construction has
been  completed.\end{proof}

\begin{lem} \label{fact2} Let $\{\tau_\xi\}_{\xi<\omega_1}$ be an uncountable family
of pairwise singular positive Borel measures on $2^\nn$.
 Then for  every finite family  $(\Gamma_i)_{i=1}^k$ of  pairwise disjoint uncountable subsets of $\omega_1$
and every $\ee>0$  there exist a family  $(\Gamma'_i)_{i=1}^k$  with  $\Gamma'_i$ an  uncountable subset
  of $\Gamma_i$  and  a family $(O_i)_{i=1}^k$ of open and pairwise disjoint subsets of  $2^\nn$ such that
$\tau_\xi(2^\nn\setminus O_i)<\ee,$
   for all $1\leq i\leq k$ and  $\xi\in \Gamma_i'$.
\end{lem}
\begin{proof} For every $\alpha<\omega_1$, we choose $(\xi_i^\alpha)_{i=1}^k\in \prod_{i=1}^k{\Gamma_i}$
 such that for every $\alpha\neq \beta$ in $\omega_1$ and every $1\leq i\leq k$, $\xi_i^\alpha\neq \xi_i^\beta$.
For each $0\leq \alpha<\omega_1$  the k-tuple
$(\tau_{\xi^i_\alpha})_{i=1}^k$ consists of pairwise singular
measures and so we may choose a k-tuple $(U^\alpha_i)_{i=1}^k$ of
clopen subsets of $2^\nn$ with the following properties: (a) For
each $i$, $\tau_{\xi^\alpha_i}(2^\nn\setminus O^\alpha_i)<\ee$ and
(b) For all $i\neq j$, $O_i^\alpha\cap O_j^\alpha=\emptyset$.

Since the family of all clopen subsets of $2^\nn$ is countable,
there is a $k$-tuple $(O_i)_i$ and an uncountable subset $\Gamma$ of $\omega_1$,
  such that for all $1\leq i\leq k$ and all $\alpha\in \Gamma$,
 $U_i^\alpha=O_i$. For each $1\leq i\leq k$, set $\Gamma_i'=\{\xi_i^\alpha:\;\alpha\in \Gamma\}$.  Then
 for each $1\leq i\leq k$ and all  $\xi\in \Gamma_i'$, $\tau_{\xi}(2^\nn\setminus O_i)<\ee$.
\end{proof}

The following lemma is the main tool used to prove Theorem \ref{t1}.

\begin{lem}\label{lemind}  Let $X$ be a Banach space with a $1$-unconditional basis,
$Y$ a closed subspace of $X$ and $T: X \to S^1$ be a
bounded  linear operator. We suppose that $\mathcal{B}_1(Y)$ contains an uncountable family
$\mathcal{G}$ such that $\mathcal{M}_{\mathcal{G}}=\{\mu_{T^{**}(x^{**})}: x^{**}\in\mathcal{G}\}$ is
non-separable. Then there is a constant $\rho >0$, such
that for every $\ee>0$ there exist a family $(y_s)_s$ of elements of $Y$
and a family $(A_s)_s$ of finite antichains of $2^{<\nn}$ such that the following are satisfied.
\begin{enumerate}
\item[(C1)] For every $s\perp t$, $A_s\perp A_t$.
\item[(C2)] For every $s\in 2^{<\nn}$, $\sum_{t\in A_s}|e_t^*(T(y_s))|>\rho$.
\item[(C3)] For every $\sigma\in 2^\nn$ and sequence $(a_k)_{k=0}^\infty$ of scalars,
\[\bigg\|\sum_{k=0}^n a_k y_{\sigma|k}\bigg\|\leq (1+\ee)\max_{0\leq k\leq n}|a_k|, \,\,\text{for all}\,\, n\geq 0.\]
\end{enumerate}
\end{lem}

\begin{proof}  Let  $\{e_i\}_{i=1}^\infty$ be an $1$-unconditional normalized basis of $X$.
Since for all $x^{**}\in \mathcal{B}_1(S^1)$ and  $\lambda\in\rr$,
$\mu_{\lambda x^{**}}=|\lambda|\mu_{x^{**}}$, we may assume  that
$\mathcal{G}\subseteq \{y^{**}\in\mathcal{B}_1(Y): \|y^{**}\|=1\}$.
By Lemma \ref{fact1},  there is  a non-separable  subset   $\{\mu_{T^{**}(y^{**}_\xi)}\}_{\xi<\omega_1}$
 of $\mathcal{M}_{\mathcal{G}}$  such that
 for all $0\leq \xi<\omega_1$,
$\mu_{T^{**}(y^{**}_\xi)}=\lambda_\xi+\tau_\xi$, and
for all $\zeta<\xi$, $\tau_{\zeta}\perp \tau_{\xi}$. By passing to
a further uncountable subset and relabeling, we may also assume that there is
$\rho_0>0$ such that $\|\tau_\xi\|>\rho_0$.
 We fix $\ee>0$ and a sequence
$(\ee_n)_n$ of positive real numbers with $\sum_{n=0}^\infty\ee_n<\ee/2$. We will construct the
following objects:
\begin{enumerate}
\item[(1)] a Cantor scheme $(\Gamma_s)_{s}$ of uncountable subsets
of $\omega_1$  (that is for all $s\in 2^{<\nn}$,
$\Gamma_{s^\smallfrown 0}\cup\Gamma_{s^\smallfrown 1}\subseteq
\Gamma_{s}$ and  $\Gamma_{s^\smallfrown
0}\cap\Gamma_{s^\smallfrown 1}=\emptyset$),
\item[(2)] a family
$(\xi_s)_s$ with $\xi_s\in\Gamma_s$, for all $s\in 2^{<\nn}$,
\item [(3)] A Cantor scheme
of open subsets $(U_s)_s$ of $2^\nn$, $U_s=\bigcup_{t\in B_s}V_t$, where $B_s$ is an
antichain of $2^{<\nn}$, for all $s\in 2^{<\nn}$,
\item[(4)] a family $(y_s)_s$ in $Y$,
\item[(5)] a family $(F_s)_s$ of finite subsets of $\nn$,
a sequence $(d_i)_{i\in F_s}$ of scalars, for all $s\in 2^{<\nn}$, and
\item[(6)] a family $(A_s)_s$ of finite antichains of $2^{<\nn}$,
\end{enumerate}
such that the following are satisfied.
\begin{enumerate}
\item [(i)] For every $\xi\in\Gamma_s$, $\tau_\xi(U_s)>\rho_0/2$
and $\tau_\xi(2^\nn\setminus
U_s)<\big(\sum_{i=0}^{|s|}2^{-(i+2)}\big)\rho_0$.
\item[(ii)] The element $y^{**}_{\xi_s}$ is
$w^*$-condensation point of $\{y^{**}_\xi\}_{\xi\in \Gamma_s}$
(by Proposition \ref{dual}, we identify the space $\mathcal{B}_1(X)$ with the space $[(e_i^*)_i]^*$,
endowed with the weak star topology).
\item [(iii)] For every $n\geq 1$, $s\in 2^n$, $\xi\in \Gamma_s$ and $i\in F_{s^-}$,
\[|y^{**}_\xi(e_i^*)-y^{**}_{\xi_{s^-}}(e_i^*)|<\frac{\ee_n}{\#F_{s^-}},\;\text{where}\;\; s^-=(s(1),...,s(n-1))\]
\item[(iv)] For every $i\in F_s$, $0\leq d_i\leq 1$ and for every
$\sigma\in 2^\nn$ and $n\geq 0$, $\max F_{\sigma|n}<\min F_{\sigma|n+1}$.
\item [(v)] $A_s\subseteq\bigcup_{t\in B_s}T_t$ and $\sum_{t\in A_s}|e_t^*(T(y_s))|>\rho$.
 \item[(vi)] For every $s\in 2^n$,
$\big\|y_s-\sum_{i\in F_s}d_iy^{**}_{\xi_s}(e_i^*)e_i\big\|<\ee_n$.
 \end{enumerate}
Given the above construction,   we set $\rho=\rho_0/2$ and we claim that
the families $(y_s)_s$ and $(A_s)_s$, satisfy conditions (C1), (C2) and (C3). Observe that (C1) follows from the fact that $(U_s)_s$ is a Cantor scheme and the first part of (v), while (C2) from the second part of (v). It remains verify condition (C3). So let $n\geq 1$ and $\sigma\in 2^\nn$.
Then $\Gamma_{\sigma|k+1}\subseteq
\Gamma_{\sigma|k}$, for every $k\geq 0$,
and so by (iii) we
get that

\begin{equation}\label{lemind1}
\begin{split}
&\bigg\|\sum_{k=0}^n \sum_{i\in F_{\sigma|k}}y^{**}_{\xi_{\sigma|k}}(e_i^*)e_i\bigg\|\\&=
\bigg\|\sum_{k=0}^n \sum_{i\in F_{\sigma|k}}y^{**}_{\xi_{\sigma|k}}(e_i^*)e_i-
\sum_{k=0}^{n-1} \sum_{i\in F_{\sigma|k}}y^{**}_{\xi_{\sigma|k}}(e_i^*)e_i+\sum_{k=0}^{n-1} \sum_{i\in F_{\sigma|k}}y^{**}_{\xi_{\sigma|k}}(e_i^*)e_i\bigg\|\\&\leq
\sum_{k=0}^{n-1} \sum_{i\in F_{\sigma|k}}|y^{**}_{\xi_{\sigma|k}}(e_i^*)-y^{**}_{\xi_{\sigma|n}}(e_i^*)|
+\bigg\|\sum_{k=0}^n \sum_{i\in F_{\sigma|k}}y^{**}_{\xi_{\sigma|n}}(e_i^*)e_i\bigg\|
\\&<\sum_{k=0}^{n-1}\ee_{k+1}+\|y^{**}_{\xi_{\sigma|n}}\|<\frac{\ee}{2}+1.
\end{split}
\end{equation}
As $\{e_i\}_i$ is a $1$-unconditional normalized basis of $X$,
by (iv) we have that if  $(a_k)_{k=0}^\infty$ is a sequence of scalars, then

\begin{equation}\label{lemind2}
\begin{split}
\bigg\|\sum_{k=0}^n a_k\sum_{i\in F_{\sigma|k}}d_iy^{**}_{\xi_{\sigma|k}}(e_i^*)e_i\bigg\|&\leq
\max_{0\leq k\leq n}|a_k|\bigg\|\sum_{k=0}^n \sum_{i\in F_{\sigma|k}}y^{**}_{\xi_{\sigma|k}}(e_i^*)e_i\bigg\|
\end{split}
\end{equation}
By (\ref{lemind1}), (\ref{lemind2}) and (vi) we obtain

\[\begin{split}\bigg\|\sum_{k=0}^n a_k y_{\sigma|k}\bigg\|&\leq
\sum_{k=0}^n |a_k|\bigg\| y_{\sigma|k}-\sum_{i\in F_{\sigma|k}}d_iy^{**}_{\xi_{\sigma|k}}(e_i^*)e_i\bigg\|
+\bigg\|\sum_{k=0}^n a_k\sum_{i\in F_{\sigma|k}}d_iy^{**}_{\xi_{\sigma|k}}(e_i^*)e_i\bigg\|\\&\leq \max_{0\leq k\leq n}|a_k|\bigg(\sum_{k=0}^n\ee_k\bigg)+\max_{0\leq k\leq n}|a_k|\bigg\|\sum_{k=0}^n
\sum_{i\in F_{\sigma|k}}y^{**}_{\xi_{\sigma|k}}(e_i^*)e_i\bigg\|\\&\leq(1+\ee)\max_{0\leq k\leq n}|a_k|.
\end{split}\]

 We present now the general
inductive step of the construction. Let us suppose that the
construction has been carried out for  all $s\in 2^{\leqslant n}$. For
every $s\in 2^{n}$, we define
 \begin{equation}\label{lemind3} \Gamma_{s}^1=\{\xi\in\Gamma_s:
 \;\forall i\in F_s,\,\,
|y^{**}_{\xi}(e_i^*)-y^{**}_{\xi_s}(e_i^*)|<\ee_{n+1}/\#F_s\}.\end{equation}
 Since $F_s$ is a finite subset of $\nn$,
the set $\{y^{**}_{\xi}:\xi\in\Gamma_{s}^1\}$
is a relatively  weak$^*$-open neighborhood of $y^{**}_{\xi_s}$ in
$\{y^{**}_{\xi}\}_{\xi\in \Gamma_s}$. By our inductive assumption,
$y^{**}_{\xi_s}$ is a
weak$^*$-condensation point of $\{y^{**}_{\xi}\}_{\xi\in
\Gamma_s}$ and therefore for all $s\in 2^n$ the set
$\Gamma_{s}^1$ is uncountable.
For every $s\in 2^n$, we choose $\Gamma_{s^\smallfrown 0}^1$ and
$\Gamma_{s^\smallfrown 1}^1$ uncountable subsets of $\Gamma_{s}^1$ with
$\Gamma_{s^\smallfrown 0}^1\cap\Gamma_{s^\smallfrown 1}^1=\varnothing$.
Applying  Lemma \ref{fact2} we
obtain a $2^{n+1}$-tuple $(O_s)_{s\in 2^{n+1}}$ of pairwise disjoint open
subsets of $2^\nn$ and a family $(\Gamma_s^2)_{s\in 2^{n+1}}$ such
that for each $s\in 2^{n+1}$, $\Gamma_s^2$ is an uncountable
subset of $\Gamma_{s^-}^1$ and  for all $\xi\in \Gamma_s^2$,
\begin{equation}\label{lemind4}\tau_\xi(2^\nn\setminus O_s)<\rho_0/2^{n+3}.\end{equation}
For every $s\in 2^{n+1}$ we set $U_s=O_s\cap U_{s^-}$. Notice that $U_s=\bigcup_{t\in B_s}V_t$, where $B_s$ is an
antichain of $2^{<\nn}$. Since
$\Gamma^2_{s}\subseteq \Gamma^1_{s^-}\subseteq
\Gamma_{s^-}$, using  (i), we get that for all $s\in 2^{n+1}$ and all
$\xi\in \Gamma^2_s$,
\begin{equation}\label{lemind5}\tau_\xi(2^\nn\setminus U_s)<\bigg(\sum_{i=0}^{n+1}2^{-(i+2)}\bigg)\rho_0.\end{equation}
Moreover as $(\sum_{i=0}^{n+1}2^{-(i+2)})\rho_0<\rho_0/2$ and
$\|\tau_\xi\|>\rho_0$, we get that for all
$\xi\in\Gamma_s^2$,
\begin{equation}\label{lemind6}\tau_\xi(U_s)>\rho_0/2.\end{equation}
We set $\Gamma_s=\Gamma^2_{s}$ and we choose $\xi_s$ in $\Gamma_s$, such that $y^{**}_{\xi_s}$ is
weak$^*$-condensation point of the set $\{y^{**}_{\xi}\}_{\xi\in\Gamma_s}$.
Since for all $\xi<\omega_1$, we have $\mu_{T^{**}(y^{**}_\xi)}\geq \tau_\xi$,
by Lemma \ref{fact}, we get that for every $s\in 2^{n+1}$ there exist  $y_s\in Y$, a finite subset $F_s$ of $\nn$,
a sequence of scalars $\{d_i\}_{i\in F_s}$, with $\max F_{s^-}<\min F_s$, $0\leq d_i\leq 1$ for every $i\in F_s$
and a finite antichain $A_s$ of $2^{<\nn}$ with $A_s\subseteq\bigcup_{t\in B_s}T_t$, such that
\begin{equation}\label{lemind7}
\bigg\|y_s-\sum_{i\in F_s}d_iy^{**}_{\xi_s}(e_i^*)e_i\bigg\|<\ee_{n+1}\,\, \text{and}\,\, \sum_{t\in A_s}|e_t^*(T(y_s))|>\rho_0/2.
\end{equation}
By (\ref{lemind1})-(\ref{lemind7}), the proof of the inductive step is complete.
\end{proof}

Let $(x_s)_{s\in 2^{<\nn}}$ be a family in $S^1$. We will say that $(x_s)_{s\in 2^{<\nn}}$ is
a \textit{block family} in $S^1$, if $(x_s)_{s\in 2^{<\nn}}$ is a block sequence of $(e_s)_{s\in 2^{<\nn}}$,
with the natural ordering of $2^{<\nn}$.

\begin{prop}\label{comp}
Let $(y_s)_{s\in 2^{<\nn}}$ be a block family in $S^1$ and  $(x_s^*)_{s\in 2^{<\nn}}$ a family in $(S^1)^*$
with the following properties.
\begin{enumerate}
\item[(a)] There is a constant $c>0$ such that $x_s^*(y_s)\geq c$,
for all $s\in 2^{<\nn}$.
\item[(b)] For every $s\in 2^{<\nn}$, $x_s^*=\sum_{t\in A_s}\ee_t e_t^*$, where $A_s$ is a finite antichain of $2^{<\nn}$
with $A_s\subseteq supp(y_s)$ and $\ee_t\in\{-1,1\}$, for all $t\in A_s$.
\item[(c)] For every $s\perp s'$, $A_s\perp A_{s'}$.
\item[(d)] There is a constant $C>0$ such
that for every $\sigma\in 2^\nn$, $n\geq 0$ and sequence $(a_k)_{k=0}^n$ of scalars,
$\|\sum_{k=0}^n a_k y_{\sigma|k}\|\leq C\max_{0\leq k\leq n}|a_k|$.
\end{enumerate}
Then the following hold.
\begin{enumerate}
\item[(i)] The family $(y_s)_{s\in 2^{<\nn}}$ is equivalent to the basis $(e_s)_{s\in 2^{<\nn}}$ of $S^1$.
\item[(ii)] There exists a projection $P: S^1\to [(y_s)_{s\in 2^{<\nn}}]$ with $\|P\|\leq C/c$.
\end{enumerate}
\end{prop}

\begin{proof}
(i) Since $(y_s)_{s\in 2^{<\nn}}$ is a block family in $S^1$, hence $1$-unconditionally basic,
we easily observe that $(y_s)_{s\in 2^{<\nn}}$ satisfies the
assumptions of Proposition \ref{eq}. Therefore (i) holds.

(ii) We define $P:S^1\rightarrow S^1$ by \[P(x)=\sum_{s\in 2^{<\nn}}\frac{x_s^*(x)}{x_s^*(y_s)}y_s ,\,\, x\in S^1.\]
By (b) we have that for every $s\in 2^{<\nn}$, $P(y_s)=y_s$. Let $x\in S^1$, $n\geq 0$ and $x^*\in (S^1)^*$ with
$\|x^*\|=1$. By Lemma \ref{ineq}, (with $|x_s^*(x)|/x_s^*(y_s)$ in place of $\lambda_s$ and $|x^*(y_s)|$ in place of $\alpha_s$),
there is an antichain $A$ of $ 2^{\leqslant n}$ and a
family of branches $(b_t)_{t\in A}$ of $2^{\leqslant n}$ such that
\[\bigg|\sum_{|s|\leq n}\frac{x_s^*(x)}{x_s^*(y_s)}x^*(y_s)\bigg|\leq\sum_{t\in A} \bigg(\sum_{s\in b_t}|x^*(y_s)|\bigg)\frac{|x_t^*(x)|}{x_t^*(y_t)}.\]
By (a), (b) and (c), we have that $\sum_{t\in A}|x_t^*(x)|/x_t^*(y_t)\leq \|x\|/c$
and by (d), we have that $\sum_{s\in b_t}|x^*(y_s)|\leq C$, for all $t\in A$. Hence
\[\bigg\|\sum_{|s|\leq n}\frac{x_s^*(x)}{x_s^*(y_s)} y_s\bigg\|\leq \frac{C}{c}\|x\|.\]
Since the above inequality holds for all $x\in S^1$ and $n\geq 0$, we obtain that $P$ is a bounded projection onto
$[(y_s)_{s\in 2^{<\nn}}]$, with $\|P\|\leq C/c$.
\end{proof}

The following lemma is easily proved by using a sliding hump argument.

\begin{lem}\label{lsubtree}
Let $(x_s)_s$ a family in $S^1$ such that for every $\sigma\in 2^\nn$, the sequence $(x_{\sigma|n})_n$
is weakly null. Then for every $\dd>0$, there exist a dyadic subtree $(t_s)_s$ of $2^{<\nn}$
and a block family $(w_s)_s$ in $S^1$ with the natural ordering of $2^{<\nn}$, such that
\[\sum_{s\in 2^{<\nn}}\|x_{t_s}-w_s\|<\dd.\]
\end{lem}

\begin{prop}\label{imp}
Let $X$ be a Banach space with an unconditional basis,
$Y$ a closed subspace of $X$ and $T: X \to S^1$ be a
bounded  linear operator, such that the set $\{\mu_{T^{**}(y^{**})}: y^{**}\in\mathcal{B}_1(Y) \}$
is non-separable subset of $\mathcal{M}(2^\nn)$. Then there exists a subspace $Z$ of
 $Y$ isomorphic to $S^1$ such that the restriction of $T$ on $Z$  is an isomorphism
 and $T[Z]$ is complemented in $S^1$. Moreover the space $Z$ is complemented in $X$.
\end{prop}

\begin{proof}
By passing to an equivalent norm, we may assume that the basis of $X$ is 1-unconditional.
Applying Lemma \ref{lemind}, there is a constant $\rho >0$, such
that for $0<\ee<\rho$ there exist a family $(y_s)_s$ of elements of $Y$
and a family $(A_s)_s$ of finite antichains of $2^{<\nn}$ such that the following are satisfied.
\begin{enumerate}
\item[(C1)] For every $s\perp t$, $A_s\perp A_t$.
\item[(C2)] For every $s\in 2^{<\nn}$, $\sum_{t\in A_s}|e_t^*(T(y_s))|>\rho$.
\item[(C3)] For every $\sigma\in 2^\nn$, $n\geq 0$ and sequence $(a_k)_{k=0}^n$ of scalars,
\[\bigg\|\sum_{k=0}^n a_k y_{\sigma|k}\bigg\|\leq (1+\ee)\max_{0\leq k\leq n}|a_k|.\]
\end{enumerate}
By the fact that the operator $T$ is bounded and (C3),
we have that for every $\sigma\in 2^\nn$, the sequence $(T(y_{\sigma|k}))_k$ is weakly null. Therefore by
Lemma \ref{lsubtree}, for $0<\dd<\ee$, there exist a dyadic subtree $(t_s)_s$ of $2^{<\nn}$
and a block family $(w_s)_s$ in $S^1$, such that
\begin{equation}\label{imp1}
\sum_{s\in 2^{<\nn}}\|T(y_{t_s})-w_s\|<\dd.
\end{equation}
First we will show that the family $(w_s)_s$ is equivalent to the basis  $(e_s)_{s\in 2^{<\nn}}$ of $S^1$
and the space $[(w_s)_s]$ is complemented in $S^1$. Indeed,
for every $s\in 2^{<\nn}$, let
$x^*_s=\sum_{v\in A_s}\ee_v e^*_v$, where $(\ee_v)_{v\in A_s}$ is a family of signs so that
\[x^*_s(T(y_{t_s}))=\sum_{v\in A_s}|e_v^*(T(y_{t_s}))|>\rho.\]
Without loss of generality, we may assume that $A_s\subseteq supp(w_s)$ for every $s\in 2^{<\nn}$
(otherwise, we replace $A_s$ by  $B_s=A_s\cap supp(w_s)$). Then by (\ref{imp1}),
it is easy to see that for every finite antichain $A$ of $2^{<\nn}$
\begin{equation}\label{imp2}
x^*_s(w_s)>\rho-\ee, \,\,\text{for all}\,\, s\in 2^{<\nn}.
\end{equation}
Since $(w_s)_s$ is a block family in $S^1$, by (C1) and (\ref{imp2}),
we get that for every finite antichain $A$ of $2^{<\nn}$ and family $(\lambda_s)_{s\in A}$ of scalars,
\begin{equation}\label{imp3}
\bigg\|\sum_{s\in A} \lambda_s w_s\bigg\|\geq (\rho-\ee)\sum_{s\in A} |\lambda_s|.
\end{equation}
On the other hand by (C3), for every $\sigma\in 2^\nn$, $n\geq 0$ and sequence $(a_k)_{k=0}^n$ of scalars,
\begin{equation}\label{imp4}
\bigg\|\sum_{k=0}^n a_k w_{t_{\sigma|k}}\bigg\|\leq ((1+\ee)\|T\|+\ee)\max_{0\leq k\leq n}|a_k|.
\end{equation}
By (\ref{imp2}), (\ref{imp3}) and (\ref{imp4}), we obtain that $(w_s)_s$  satisfies the
assumptions of Proposition \ref{comp}.
Hence the family $(w_s)_s$ is equivalent to the basis $(e_s)_{s\in 2^{<\nn}}$
and the space $[(w_s)_s]$ is complemented in $S^1$. Therefore by (\ref{imp1}), for $\delta>0$ sufficiently small, the family $(T(y_{t_s}))_s$ is equivalent to the basis $(e_s)_{s\in 2^{<\nn}}$ and
the subspace $T[Z]$ is complemented in $S^1$, where $Z=[(y_{t_s})_s]$.  Now by (C3), as in the proof of
Proposition \ref{eq}, $((y_{t_s})_s$ has an upper $S^1$-estimate, which gives that
the family $((y_{t_s})_s$ is also equivalent to the basis $(e_s)_{s\in 2^{<\nn}}$.

We now proceed to show that the space $Z$ is complemented in $X$.
Let $P$ be a bounded projection from $S^1$ onto $T[Z]$. We set $Q=T^{-1}PT$. Then
$Q$ is a a bounded projection from $X$ onto $Z$. This completes the proof of the proposition.
\end{proof}

\begin{lem}\label{ldom}
Let $x^{**}\in\mathcal{B}_1(S^1)$. Then for every $\ee>0$ there exists $n\in\nn$ such that
for every finite antichain $A$ of $2^{<\nn}$,
\[\|x^{**}|\cup_{s\in A} T^n_s\|<\mu_{x^{**}}(\cup_{s\in A}V_s)+\ee.\]
Therefore, for every finite antichain $A$ of $2^{<\nn}$
with $\min\{|s|: s\in A\}\geq n$,
\[\sum_{s\in A}|x^{**}(e_s^*)|<\mu_{x^{**}}(\cup_{s\in A}V_s)+\ee.\]
\end{lem}

\begin{proof}Let $\ee>0$ and $A$ be a finite antichain of $2^{<\nn}$.
By the definition of $\mu_{x^{**}}$, there exists $n\in\nn$ such that
\begin{equation}\label{ldom1}
\|x^{**}|T^n_{\varnothing}\|<\mu_{x^{**}}(2^\nn)+\ee.
\end{equation}
We choose an antichain $A'$ of $2^{<\nn}$ such that $A\cup A'$ is a finite maximal antichain
of $2^{<\nn}$ and $A\cap A'=\varnothing$.
Then
\begin{equation}\label{ldom2}
\|x^{**}|\cup_{s\in A} T^n_s\|+\|x^{**}|\cup_{s\in A'} T^n_s\|=\|x^{**}| T^n_{\varnothing}\|
\end{equation} and
\begin{equation}\label{ldom3}
\mu_{x^{**}}(2^\nn)=\mu_{x^{**}}(\cup_{s\in A}V_s)+\mu_{x^{**}}(\cup_{s\in A'}V_s)
\end{equation}
By (\ref{ldom1}), (\ref{ldom2}) and (\ref{ldom3}), we get that for every $n\geq l$,
\[\begin{split}\|x^{**}|\cup_{s\in A} T^n_s\|+\|x^{**}|\cup_{s\in A'} T^n_s\|&<\mu_{x^{**}}(\cup_{s\in A}V_s)+\mu_{x^{**}}(\cup_{s\in A'}V_s)+\ee\\&
\leq\mu_{x^{**}}(\cup_{s\in A}V_s)+\|x^{**}|\cup_{s\in A'} T^n_s\|+\ee.\end{split}\]
Therefore, we obtain that $\|x^{**}|\cup_{s\in A} T^n_s\|<\mu_{x^{**}}(\cup_{s\in A}V_s)+\ee$.
\end{proof}

\begin{prop}\label{nonsep}
Let $(y_s)_{s\in 2^{<\nn}}$ be a block family in $S^1$ with the following properties.
\begin{enumerate}
\item[(i)] There is a constant $\rho>0$ such that
$\|\sum_{|s|=n}\lambda_s y_s\|\geq\rho\sum_{|s|=n}|\lambda_s|$, for every $n\geq 0$
and family $(\lambda_s)_{|s|=n}$ of scalars.
\item[(ii)] For every $\sigma\in 2^\nn$, the sequence $(y_{\sigma|k})_{k=0}^\infty$ is equivalent to the basis of $c_0$.
\end{enumerate}
Then the set $\{\mu_{y^{**}_{\sigma}}: \sigma\in 2^\nn\}$ is a non-separable subset of $\mathcal{M}(2^\nn)$, where
for every $\sigma\in 2^\nn$,  $y^{**}_{\sigma}=w^*-\lim_n\sum_{k=0}^n y_{t_{\sigma|k}}$.
\end{prop}

\begin{proof}
If the set $\{\mu_{y^{**}_{\sigma}}: \sigma\in 2^\nn\}$ is separable, we can choose a norm-condensation point
$\mu\in \mathcal{M}(2^\nn)$ of $\{\mu_{y^{**}_{\sigma}}: \sigma\in 2^\nn\}$. Also fix $m\in\nn$ and $\ee>0$.
Then for uncountably many $\sigma\in 2^\nn$, we have that
\begin{equation}\label{nonsep1}
\|\mu_{y^{**}_{\sigma}}-\mu\|\leq \ee/m.
\end{equation}
Let $\sigma_1,...,\sigma_m\in 2^\nn$ satisfying (\ref{nonsep1}) and $n_0\in\nn$
be such that for all $n\geq n_0$ and $1\leq i<j\leq m$, $\sigma_i|n\perp\sigma_j|n$. By Lemma \ref{ldom},
there is $k\geq n_0$ so that
\begin{equation}\label{nonsep2}
\sum_{s\in A}|y^{**}_{\sigma_i}(e_s^*)|<\mu_{y^{**}_{\sigma_i}}(\cup_{s\in A}V_s)+\ee/m,
\end{equation}
for every finite antichain $A$ of $2^{<\nn}$
with $\min\{|s|: s\in A\}\geq k$ and each $i=1,...,m$.
We choose an  antichain $A$ such that
$\|\sum_{i=1}^m y_{\sigma_i|k}\|=\sum_{s\in A}|e^*_s(\sum_{i=1}^m y_{\sigma_i|k})|$
and we set $A_i=A\cap supp(y_{\sigma_i|k})$. Since  $(y_s)_{s\in 2^{<\nn}}$ is a block family in $S^1$,
we may assume that $\min\{|s|: s\in A_i\}\geq k$. Then by (i), (\ref{nonsep1}) and (\ref{nonsep2}),
we get that
\[\begin{split}
m\rho\leq\bigg\|\sum_{i=1}^m y_{\sigma_i|k}\bigg\|&=
\sum_{s\in A}\bigg|e^*_s\bigg(\sum_{i=1}^m y_{\sigma_i|k}\bigg)\bigg|=
\sum_{i=1}^m\sum_{s\in A_i}|y^{**}_{\sigma_i}(e_s^*)|\\&<\sum_{i=1}^m\mu_{y^{**}_{\sigma_i}}(\cup_{s\in A_i}V_s)+\ee
=\mu(\cup_{s\in\cup_{i=1}^m A_i}V_s)+\ee\leq\|\mu\|+\ee,
\end{split}\]
a contradiction for $\ee$ sufficiently small.
\end{proof}

\begin{lem}\label{lequal}
Let $(x_n)_n$, $(y_n)_n$ be sequences in $S^1$ which are equivalent to the $c_0$-basis
and $\sum_n\|x_n-y_n\|<\infty$.
We set \[x^{**}=w^*-\lim_n\sum_{k=0}^nx_k\,\, \text{and}\,\, y^{**}=w^*-\lim_n\sum_{k=0}^ny_k.\]
Then $\mu_{x^{**}}=\mu_{y^{**}}$.
\end{lem}

\begin{proof}
For every $n\in\nn$, we set \[x^{**}_{|\geq n}=w^*-\lim_m\sum_{k=n}^m x_k\,\, \text{and}\,\,
y^{**}_{|\geq n}=w^*-\lim_m\sum_{k=n}^m y_k.\]
Since $x^{**}-x^{**}_{|\geq n}\in S^1$ and $y^{**}-y^{**}_{|\geq n}\in S^1$, we have that
$\mu_{x^{**}}=\mu_{x^{**}_{|\geq n}}$ and $\mu_{y^{**}}=\mu_{y^{**}_{|\geq n}}$. Then
from the weak star lower semicontinuity of the second dual norm we get that
\[\begin{split}\|\mu_{x^{**}}-\mu_{y^{**}}\|=\|\mu_{x^{**}_{|\geq n}}-\mu_{y^{**}_{|\geq n}}\|\leq
\|x^{**}_{|\geq n}-y^{**}_{|\geq n}\|&\leq\liminf_m\bigg\|\sum_{k=n}^m x_k-\sum_{k=n}^m y_k\bigg\|
\\&\leq\sum_{k=n}^\infty\|x_k-y_k\|.\end{split}\]
Therefore, letting $n\rightarrow\infty$, $\|\mu_{x^{**}}-\mu_{y^{**}}\|=0$ and so  $\mu_{x^{**}}=\mu_{y^{**}}$.
\end{proof}

\begin{prop}\label{conv}
Let $X$ be a Banach space with an unconditional basis,
$T: X \to S^1$ be a bounded  linear operator such that there exists a subspace $Y$ of
$X$ isomorphic to $S^1$ and the restriction of $T$ on $Y$ is an isomorphism.
Then the set $\{\mu_{T^{**}(y^{**})}: y^{**}\in\mathcal{B}_1(Y)\}$
is a non-separable subset of $\mathcal{M}(2^\nn)$.
\end{prop}

\begin{proof}
Since $Y$ is isomorphic to $S^1$, there is a family $(y_s)_{s\in 2^\nn}$ in $Y$
equivalent to the $S^1$-basis with $Y=[(y_s)_s]$. By Lemma \ref{lsubtree},
there exist a dyadic subtree $(t_s)_s$ of $2^{<\nn}$
and a block family $(w_s)_s$ in $S^1$, such that
\begin{equation}\label{conv1}
\sum_{s\in 2^{<\nn}}\|T(y_{t_s})-w_s\|<\dd/2.
\end{equation}
By (\ref{conv1}) we have that the family $(w_s)_s$ is equivalent to the $S^1$-basis.
For every $\sigma\in 2^\nn$, we set $y^{**}_{\sigma}=w^*-\lim_n\sum_{k=0}^n y_{t_{\sigma|k}}$ and
$w^{**}_{\sigma}=w^*-\lim_n\sum_{k=0}^n w_{\sigma|k}$. Applying Proposition \ref{nonsep}
we obtain that the set $\{\mu_{w^{**}_{\sigma}}: \sigma\in 2^\nn\}$ is a non-separable subset of $\mathcal{M}(2^\nn)$.
Since $T^{**}(y^{**}_{\sigma})=w^*-\lim_n\sum_{k=0}^n T(y_{t_{\sigma|k}})$, by (\ref{conv1}) and Lemma \ref{lequal},
we get that $\mu_{w^{**}_{\sigma}}= \mu_{T^{**}(y^{**}_{\sigma})}$, for every $\sigma\in 2^\nn$. Therefore the set
$\{\mu_{T^{**}(y^{**})}: y^{**}\in\mathcal{B}_1(Y)\}$
is non-separable.
\end{proof}

Observe that Proposition \ref{imp} and  Proposition
\ref{conv} yield Theorem \ref{t1} from the introduction.

\begin{rem}
The proof of Theorem \ref{t1} actually yields a slightly stronger result, in particular the conclusion holds if $X$ is assumed to be a subspace of a space with an unconditional basis.
\end{rem}

\begin{cor}\label{corS}
Let  $Y$ be a closed subspace of $S^1$. Then the set
$\mathcal{M}_{\mathcal{B}_1(Y)}=\{\mu_{y^{**}}: y^{**}\in\mathcal{B}_1(Y) \}$
is non-separable subset of $\mathcal{M}(2^\nn)$, if and only if, there exists a subspace $Z$ of
$Y$ isomorphic to $S^1$ and complemented in $S^1$.
 \end{cor}

\begin{proof}
Let $I:  S^1 \to S^1$ the identity operator.
Then we observe that $I^{**}[\mathcal{B}_1(Y)]=\mathcal{B}_1(Y)$.
By Theorem \ref{t1} the conclusion follows.
\end{proof}

The following corollary is an immediate consequence of Corollary \ref{corS}.

\begin{cor}
Let $Y$ be a closed subspace of $S^1$ which is isomorphic to $S^1$.
Then $Y$ contains a subspace $Z$ isomorphic to $S^1$ which is complemented in $S^1$.
\end{cor}

To state one more consequence of the above theorem we will need the
following.
\begin{lem}\label{Pel}
Let $Y$ be a subspace of $S^1$. Suppose that $S^1$
contains a complemented copy of $Y$ and that $Y$ contains a
complemented copy of $S^1$. Then $Y$ is isomorphic to $S^1$.
\end{lem}

\begin{proof}
Notice that $S^1\approx c_0\oplus(S^1\oplus S^1\oplus...)_{\ell_1}\approx
c_0\oplus(S^1\oplus S^1\oplus...)_{\ell_1}\oplus (S^1\oplus
S^1\oplus...)_{\ell_1}\approx S^1\oplus(S^1\oplus S^1\oplus...)_{\ell_1}\approx
(S^1\oplus S^1\oplus...)_{\ell_1}$. The
result now follows by applying the Pelczynski decomposition method
\cite{LT}.
\end{proof}

\begin{cor}\label{c28}
Let $Y$ be a complemented subspace of $S^1$ such that the set
$\{\mu_{y^{**}}: y^{**}\in\mathcal{B}_1(Y) \}$ is a non-separable subset of $\mathcal{M}(2^\nn)$.
Then $Y$ is isomorphic to $S^1$.
\end{cor}
\begin{proof} Since $Y$ is a subspace of $S^1$ and $\{\mu_{y^{**}}: y^{**}\in\mathcal{B}_1(Y) \}$ is non-separable,
by Corollary \ref{corS} we have that $Y$
contains
a complemented copy of $S^1$. Since $Y$ is complemented in $S^1$, by
Lemma \ref{Pel} we have that $Y$ is isomorphic to $S^1$.
\end{proof}

\begin{rem}\label{Ro29}
It follows from a standard argument using Rosenthal's lemma \cite{Ro}, that whenever $(X_i)_i$ is a sequence of Banach spaces and $X = (\sum_{i=1}^\infty\oplus X_i)_{\ell_1}$ then $\mathcal{B}_1(X) = (\sum_{i=1}^\infty\oplus \mathcal{B}_1(X_i))_{\ell_1}$. More precisely, for every $x^{**}\in\mathcal{B}_1(X)$, there exists a unique sequence $(x_i^{**})_i$ with $x_i^{**}\in \mathcal{B}_1(X_i)$ for all $i\in\nn$ and $\sum_i\|x_i^{**}\| < \infty$, so that $x^{**} = \sum_{i=1}^\infty x_i^{**}$.
\end{rem}

\begin{cor}\label{Cor31}
Suppose $S^1$ is isomorphic to an $\ell_1$ sum of a sequence of Banach spaces $X_i$.
Then there is an $j$ such that $X_j$ is isomorphic of $S^1$.
\end{cor}

\begin{proof}
Let $X=(\sum_{i=1}^\infty\oplus X_i)_{\ell_1}$ and assume that $X$ is isomorphic to $S^1$. If we identify $X$ with $S^1$, by Corollary \ref{c28},
it is enough to prove that there is an $j$ such that the set
$\{\mu_{x^{**}}: x^{**}\in\mathcal{B}_1(X_j) \}$ is a non-separable subset of $\mathcal{M}(2^\nn)$.
Assume that the claim is not true. Then for every $i$, the set $\{\mu_{x^{**}}: x^{**}\in\mathcal{B}_1(X_i) \}$
is separable. Therefore for every $i$ there is $\mu_i\in\mathcal{M}^+(2^\nn) $ such that $\mu_{x^{**}}<<\mu_i$,
for all $x^{**}\in\mathcal{B}_1(X_i)$. We set $\mu=\sum_{i=1}^\infty\frac{\mu_i}{2^i(1+\|\mu_i\|)}$.
By Proposition \ref{basprop} (ii) and (iv), we obtain that $\mu_{x^{**}}<<\mu$, for every $x^{**}$ in $(\sum_i\oplus\mathcal{B}_1(X_i))_{\ell_1}$, which by Remark is $\mathcal{B}_1(S^1)$.
Hence the set $\{\mu_{x^{**}}: x^{**}\in\mathcal{B}_1(S^1)\}$ is separable, a contradiction.
\end{proof}

\begin{rem}
In \cite{BO2} H. Bang and E. Odell showed that $S^1$ is primary, that is whenever $S^1 = X\oplus Y$, then either $X$ or $Y$ is isomorphic to $S^1$. We note that this fact also follows from Corollary \ref{Cor31}.
\end{rem}

The following is an immediate consequence of the main result from \cite{BV}, where actually a stronger statement is proven.

\begin{prop}\label{den}
Let D be a subset of $2^{<\nn}$ such that
\[\limsup_{n\rightarrow\infty}\frac{\#(D\cap \{s\in 2^{<\nn}: |s|=n\} )}{2^n}>0.\]
Then there is a regular dyadic subtree $(t_s)_s$ of $2^{<\nn}$ which is contained in D.
\end{prop}

The next proposition can be found in \cite[Proposition 2.2, page 161]{BR}.

\begin{prop}\label{G}
Let $X$ and $Y$ be Banach spaces and $T: X \to Y$ be a $G_\delta$-embedding.
If  $X$ is isomorphic to $c_0$, T is an isomorphism.
\end{prop}

\begin{thm}\label{TG}
Let $X$ be a Banach space. If $S^1$ $G_\delta$-embeds in $X$ and $X$ $G_\delta$-embeds in $S^1$,
then $S^1$ complementably embeds in $X$.
\end{thm}

\begin{proof}
Let $W: S^1 \to X$, $R: X \to S^1$ be $G_\delta$-embeddings and $T=RW: S^1 \to S^1$.
By Lemma \ref{lsubtree}, passing to the dyadic subtree, we may assume that
$(T(e_s))_s$ is a block family in $S^1$. For every $n\in\nn$, we set
$x_n=\sum_{|s|=n}\frac{1}{2^n}e_s$. Notice that the space $[(x_n)_{n\in\nn}]$ is
isometric to $c_0$. Then by Proposition \ref{G}, the restriction of $T$ on $[(x_n)_{n\in\nn}]$  is an isomorphism.
Hence there is $\rho>0$ with $2\|T\|>\rho$, such that $\|T(x_n)\|\geq\rho$, for all $n\in\nn$. For each $n$, we choose a norm one
functional $x^*_n\in (S^1)^*$ such that $x^*_n(T(x_n))=\|T(x_n)\|$ and we set
\[A_n=\{s\in 2^{<\nn}: x^*_n(T(e_s))\geq\rho/2,\,\, |s|=n\}\,\,\text{and}\,\,D=\cup_{n\in\nn}A_n.\]
Then we have that
\[\begin{split}
\rho\leq x^*_n(T(x_n))=\frac{1}{2^n}\sum_{|s|=n}x^*_n(T(e_s))&=\frac{1}{2^n}\sum_{s\in A_n}x^*_n(T(e_s))+
\frac{1}{2^n}\sum_{s\in A_n^c}x^*_n(T(e_s))\\&\leq\frac{1}{2^n}\bigg(\#A_n\|T\|+(2^n-\#A_n)\frac{\rho}{2}\bigg).
\end{split}\]
By the above inequality we get that $\frac{\#A_n}{2^n}\geq\frac{\rho}{2\|T\|-\rho}$. Hence,
\begin{equation}\label{TG1}
\frac{\#(D\cap \{s\in 2^{<\nn}: |s|=n\} )}{2^n}=\frac{\#A_n}{2^n}\geq\frac{\rho}{2\|T\|-\rho}>0,\,\,\text{for all}\,\, n\in\nn.
\end{equation}
By (\ref{TG1}) and Proposition \ref{den}, there is a regular dyadic subtree $(t_s)_s$ of $2^{<\nn}$ which is contained in D.
Since the family $(T(e_s))_s$ is block, we easily observe that for every $n\geq 0$ and family $(\lambda_s)_{|s|=n}$ of scalars,
$\|\sum_{|s|=n}\lambda_s T(e_{t_s})\|\geq\rho/2\sum_{|s|=n}|\lambda_s|.$
Therefore, the assumptions of Proposition \ref{nonsep} are fulfilled and so the set
$\{\mu_{T^{**}(\sigma^{**})}: \sigma\in 2^\nn\}$ is non-separable subset of $\mathcal{M}(2^\nn)$, where
for every $\sigma\in 2^\nn$,  $\sigma^{**}=w^*-\lim_n\sum_{k=0}^n e_{t_{\sigma|k}}$. Now by Theorem \ref{t1} the conclusion follows.
\end{proof}

As a consequence of Theorem \ref{TG}, we obtain the following
\begin{cor}
Let $X$ be a closed subspace of $S^1$. If $S^1$ $G_\delta$-embeds in $X$,
then $S^1$ complementably embeds in $X$.
\end{cor}

\begin{rem}
Note that a positive answer to the following question, implies that problem 2 from the introduction has a positive answer as well.
Let $T:S^1\to S^1$ be an operator and $X$ be a reflexive subspace
of $S^1$ such that the  restriction of $T$ on $X$  is an isomorphism.
Is it true that the set  $\{\mu_{T^{**}(x^{**})}: x^{**}\in\mathcal{B}_1(S^1) \}$
is a non-separable subset of $\mathcal{M}(2^\nn)$?
\end{rem}

\section*{Acknowledgement}
I would like to thank Professor S. A. Argyros and Professor I. Gasparis, whose crucial observations helped improve the content as well as the presentation of the results of this paper.

\end{document}